\documentclass[12pt]{amsart}
\usepackage{amsaddr}
\usepackage{url}
\usepackage{geometry}
\usepackage{xcolor}
\usepackage{graphicx}
\usepackage{float}
\usepackage{tikz}
\usepackage{cite}
\usetikzlibrary{arrows,shapes,positioning}
\usetikzlibrary{decorations.markings}
\usepackage{amssymb,amsmath,amsthm}
\numberwithin{equation}{section}
\allowdisplaybreaks
\newtheorem{thm}{Theorem}[section]
\newtheorem{lemma}[thm]{Lemma}
\newtheorem{remark}[thm]{Remark}

\newtheorem{definition}[thm]{Definition}

\renewcommand{\(}{\left(}
\renewcommand{\)}{\right)}

\newcommand{\E}{{\rm E}}
\newcommand{\tr}{{\rm tr}}
\newcommand{\mb}{\mathbf}
\newcommand{\du}{\circ}
\newcommand{\di}{{\rm Diag}}

\begin{document}

\makeatletter
\renewcommand{\email}[2][]{%
  \ifx\emails\@empty\relax\else{\g@addto@macro\emails{,\space}}\fi%
  \@ifnotempty{#1}{\g@addto@macro\emails{\textrm{(#1)}\space}}%
  \g@addto@macro\emails{#2}%
}
\makeatother

\title[Homoscedasticity test for both low and high-dimensional regressions]{Homoscedasticity tests for both low and high-dimensional fixed design regressions}

\author{Zhidong Bai}
\address{KLASMOE and School of Mathematics and Statistics, Northeast Normal University, Changchun, P.R.C., 130024.}
\email{baizd@nenu.edu.cn}
\thanks{Zhidong Bai is partially supported by a grant NSF China 11571067}

\author{Guangming Pan}
\address{School of Physical and Mathematical Sciences, Nanyang Technological University,
Singapore, 637371}
\email{gmpan@ntu.edu.sg}
\thanks{G. M. Pan was partially supported by a MOE Tier 2 grant 2014-T2-2-060
and by a MOE Tier 1 Grant RG25/14 at the Nanyang Technological University, Singapore.}

\author{Yanqing Yin}
\address{KLASMOE and School of Mathematics and Statistics, Northeast Normal University, Changchun, P.R.C., 130024.}
\email[Corresponding author]{yinyq799@nenu.edu.cn}
\thanks{Yanqing Yin was partially supported by a project of China Scholarship Council}


\keywords{Breusch and Pagan test, White's test, heteroscedasticity, homoscedasticity, high-dimensional
regression, design matrix}

\maketitle
\begin{abstract}
This paper is to prove the asymptotic normality of a statistic for detecting the existence of heteroscedasticity for linear regression models without assuming randomness of covariates when the sample size $n$ tends to infinity and the number of covariates $p$ is either fixed or tends to infinity. Moreover our approach indicates that its asymptotic normality holds even without homoscedasticity.
\end{abstract}

\section{Introduction}

\subsection{A brief review of homoscedasticity test}

Consider the classical multivariate linear regression model of $p$ covariates
\begin{align}
  y_i=\mb x_i\mb \beta+\mb \varepsilon_i,\ \ \ \ \ i=1,2,\cdots,n,
\end{align}
where $y_i$ is the response variable, $\mb x_i=(x_{i,1},x_{i,2},\cdots,x_{i,p})$ is the $p$-dimensional covariates, $\beta=\(\beta_1,\beta_2,\cdots,\beta_p\)'$ is the $p$ dimensional regression coefficient vector and $\varepsilon_i$ is the independent random errors obey the same distribution with zero mean and variance $\sigma_i^2$. In most applications of the linear regression models the homoscedasticity is a very important assumption. Without it, the loss in efficiency in using ordinary least squares (OLS) may be substantial and even worse, the biases in estimated standard errors may lead to invalid inferences. Thus, it is very important to examine  the homoscedasticity. Formally, we need to test the hypothesis
\begin{equation}\label{a1}
H_0: \ \sigma_1^2=\sigma_2^2=\cdots=\sigma_n^2=\sigma^2,
\end{equation}
where $\sigma^2$ is a positive constant.

In the literature there are a lot of work considering this hypothesis test when the dimension $p$ is fixed. Indeed, many popular tests have been proposed. For example Breusch and Pagan \cite{breusch1979simple} and White \cite{white1980heteroskedasticity} proposed statistics to investigate the relationship between the estimated errors and the covariates in economics. While in statistics, Dette and Munk \cite{dette1998estimating}, Glejser \cite{glejser1969new}, Harrison and McCabe \cite{harrison1979test}, Cook
and Weisberg \cite{cook1983diagnostics}, Azzalini and Bowman\cite{azzalini1993use} proposed nonparametric statistics to conduct the hypothesis. One may refer to Li and Yao \cite{li2015homoscedasticity} for more details in this regard.

The development of computer science makes it possible for people to collect and deal with high-dimensional data. As a consequence, 
high-dimensional linear regression problems are becoming more and more common due to widely available covariates. Note that the above mentioned tests are all developed under the low-dimensional framework when the dimension $p$ is fixed and the sample size $n$ tends to infinity.  

In Li and Yao's paper, they proposed two test statistics in the high dimensional setting by using the regression residuals. The first statistic uses the idea of likelihood ratio and the second one uses the idea that ``the departure of a sequence of numbers from a constant can be efficiently assessed by its coefficient of variation", which is closely related to John's idea \cite{john1971some}. By assuming that the distribution of the covariates is $\mb N(\mb 0, \mb I_p)$ and that the error obey the normal distribution, the ``coefficient of variation" statistic turns out to be a function of residuals.  But its asymptotic distribution missed some part as indicated from the proof of Lemma 1 in \cite{li2015homoscedasticity} even in the random design. 

 The aim of this paper is to establish central limit theorem for the ``coefficient of variation" statistic without assuming randomness of the covariates by using the information in the projection matrix (the hat matrix).
 This ensures that the test works when the design matrix is both fixed and random. More importantly we prove that the asymptotic normality of this statistics holds even without homoscedasticity. That assures a high power of this test.

The structure of this paper is as follows. Section 2 is to give our main theorem and some simulation results, as well as two real data analysis. Some calculations and the proof of the asymptotic normality are presented in Section 3.

\section{Main Theorem, Simulation Results and Real Data Analysis}

\subsection{The Main Theorem}
Suppose that the parameter vector $\beta$ is estimated by the OLS estimator $$\hat{ {\beta}}=\(\mb X'\mb X\)^{-1}\mb X'\mb Y.$$
Denote the residuals by $$\hat{\mb \varepsilon}=\(\hat{\mb \varepsilon_1},\hat{\mb \varepsilon_2},\cdots,\hat{\mb \varepsilon_n}\)'=\mb Y-\mb X\hat \beta=\mb P\varepsilon,$$ with $\mb P=(p_{ij})_{n\times n}=\mb I_n-\mb X(\mb X'\mb X)^{-1}\mb X'$ and $\varepsilon=\(\varepsilon_1,\varepsilon_2,\cdots,\varepsilon_n\)'$. Let $\mb D$ be an $n\times n$ diagonal matrix with its $i$-th diagonal entry being $\sigma_i$, set $\mb A=(a_{ij})_{n\times n}=\mb P\mb D$ and let $\xi=\(\xi_1,\xi_2,\cdots,\xi_n\)'$ stand for a standard $n$ dimensional random vector whose entries obey the same distribution with $\varepsilon$. It follows that the distribution of $\hat \varepsilon$ is the same as that of $\mb A\xi$. In the following, we use $\di\(\mb B\)=\(b_{1,1},b_{2,2},\cdots,b_{n,n}\)'$ to stand for the vector formed by the diagonal entries of $\mb B$ and $\di'\(\mb B\)$ as its transpose, use $\mb D_{\mb B}$ stand for the diagonal matrix of $\mb B$, and use $\mb 1$ stand for the vector $\(1,1,\cdots,1\)'$.

Consider the following statistic
\begin{equation}\label{a2}
\mb T=\frac{\sum_{i=1}^n\(\hat{\varepsilon_i}^2-\frac{1}{n}\sum_{i=1}^n\hat{\varepsilon_i}^2\)^2}{\frac{1}{n}\(\sum_{i=1}^n\hat{\varepsilon_i}^2\)^2}.
\end{equation}
  We below use $\mb A \du \mb B$ to denote the Hadamard product of two matrices $\mb A$ and $\mb B$ and use $\mb A ^{\du k}$ to denote the Hadamard product of k $\mb A$.
\begin{thm} \label{th1}
  Under the condition that the distribution of $\varepsilon_1$ is symmetric, $\E|\varepsilon_1|^8\leq \infty$ and $p/n\to y\in [0,1)$ as $n\to \infty$, we have
  $$\frac{\mb T-a}{\sqrt b}\stackrel{d}{\longrightarrow}\mb N(0,1)$$
\end{thm}
 where $a$, $b$ are determined by $n$, $p$ and $\mb A$. Under $H_0$, we further have
$$a=\(\frac{n\(3\tr\(\mb P\du \mb P\)+\nu_4\tr(\mb P\du \mb P)^2\)}{\(\(n-p\)^2+2\(n-p\)+\nu_4\tr (\mb P\du \mb P)\)}-1\),\ b=\Delta'\Theta\Delta,$$ where
$$\Delta'=(\frac{n}{\(\(n-p\)^2+2\(n-p\)+\nu_4\tr (\mb P\du \mb P)\)},-\frac{n^2\(3\tr\(\mb P\du \mb P\)+\nu_4\tr(\mb P\du \mb P)^2\)}{{\(\(n-p\)^2+2\(n-p\)+\nu_4\tr (\mb P\du \mb P)\)}^2})$$ and
$$\Theta=\left(
 \begin{array}{cc}
  \Theta_{11} & \Theta_{12} \\
   \Theta_{21} & \Theta_{22} \\
   \end{array}
   \right),
$$
where
\begin{align}
\Theta_{11}=&72\di'(\mb P) \(\mb P\du\mb P\)\di (\mb P)+24\tr\(\mb P\du\mb P\)^2\\\notag
&+\nu_4\(96\tr\mb P \mb {D_P} \mb P \mb P^{\du 3}+72\tr(\mb P\du \mb P)^3+36\di'(\mb P) \(\mb P\du\mb P\)^2\di (\mb P) \)\\\notag
&+\nu^2_4\( 18\tr(\mb P\du\mb P)^4+16\tr (\mb P^{\du 3}\mb P)^2)\)\\\notag
&+\nu_6\(12\tr\(\(\mb P\mb D_{\mb P}\mb P \)\du\(\mb P^{\du 2}\mb P^{\du 2}\)\)+16\tr \mb P \mb P^{\du 3}\mb P^{\du 3}\)+\nu_8\mb 1'(\mb P^{\du 4}\mb P^{\du 4})\mb 1,
\end{align}

\begin{align}
\Theta_{22}=\frac{8\(n-p\)^3+4\nu_4\(n-p\)^2\tr(\mb P\du\mb P)}{n^2},
\end{align}

\begin{align}
&\Theta_{12}=\Theta_{21}\\\notag
=&\frac{\(n-p\)}{n}\(24\tr\(\mb P\du \mb P\)+16\nu_4\tr(\mb P \mb P^{\du 3})+12\nu_4\tr\(\(\mb P\mb D_{\mb p}\mb P\)\du \mb P\)
+2\nu_6[\di(\mb P)'(\mb P^{\du 4})\mb 1]\),
\end{align}

$\nu_4=M_4-3$ , $\nu_6=M_6-15 M_4+30$ and $\nu_8=M_8-28 M_6-35M_4^2+420M_4-630$ are the corresponding cumulants of random variable $\varepsilon_1$.

\begin{remark}
The existence of the 8-th moment is necessary because it determines the asymptotic variance of the statistic.
\end{remark}

\begin{remark}
The explicit expressions of $a$ and $b$ are given in Theorem \ref{th1} under $H_0$. However the explicit expressions of $a$ and $b$ are quite complicated under $H_1$. Nevertheless one may obtain them from (\ref{e1})-(\ref{e2}) and (\ref{t10})-(\ref{ct12}) below.
\end{remark}
\begin{remark}
  In Li and Yao's paper, under the condition that the distribution of $\varepsilon$ is normal, they also did some simulations when the design matrices are non-Gaussian. Specifically speaking, they also investigated the test when the entries of design matrices are drawn from gamma distribution $G(2,2)$ and uniform distribution $U(0,1)$ respectively. There is no significant difference in terms of size and power between these two non-normal designs and the normal design. This seems that the proposed test is robust against the form of the distribution of the design matrix. But according to our main theorem, it is not always the case. In our main theorem, one can find that when the error $\varepsilon$ obey the normal distribution, under $H_0$ and given $p$ and $n$, the expectation of the statistics is only determined by $\tr(\mb P\du \mb P)$. We conduct  some simulations to investigate the influence of the distribution of the design matrix on this term when $n=1000$ and $p=200$. The simulation results are presented in table \ref{table1}.
\begin{table}[!hbp]
  \center
    \begin{tabular}{|c|c|c|c|c|c|}
       \hline
        & $N(0,1)$ & $G(2,2)$ & $U(0,1)$ & $F(1,2)$ & $exp(N(5,3))/100$ \\
       \hline
       $\tr(\mb P\du \mb P)$ & 640.3 & 640.7 & 640.2 & 712.5 & 708.3 \\
       \hline
     \end{tabular}
  \caption{The value of $\tr(\mb P\du \mb P)$ corresponding to different design distributions}\label{table1}
  \end{table}
  It suggests that even if the entries of the design matrix are drawn from some common distribution, the expectation of the statistics may deviate far from that of the normal case. This will cause a wrong test result. Moreover, even in the normal case, our result is more accurate since we do not use any approximate value in the mean of the statistic $T$.
\end{remark}

\begin{remark}
  Let's take an example to explain why this test works. For convenient, suppose that $\varepsilon_1$ obey the  normal distribution. From the calculation in Section \ref{exp} we know that the expectation of the statistic $\mb T$ defined in (\ref{a2})
  can be represented as $$\E \mb T=\frac{3n\sum_{i=1}^np_{ii}^2\sigma_i^4}{(\sum_{i=1}^np_{ii}\sigma_i^2)^2}-1+o(1).$$
 Now assume that $p_{ii}=\frac{n-p}{n}$ for all $i=1,\cdots,n$. Moreover, without loss of generality, suppose that $\sigma_1=\cdots=\sigma_n=1$ under $H_0$ so that we get $\E \mb T\to 2$ as $n \to \infty$. However, when $\sigma_1=\cdots=\sigma_{[n/2]}=1$ and $\sigma_{[n/2]+1}=\cdots=\sigma_{n}=2$, one may obtain $\E \mb T\to 3.08$ as $n\to \infty$. Since  ${\rm Var}(\mb T)=O(n^{-1})$ this ensures a high power as long as $n$ is large enough.
\end{remark}

\subsection{Some simulation results}

We next conduct some simulation results to investigate the performance of our test statistics.
Firstly, we consider the condition when the random error obey the normal distribution.
Table \ref{table2} shows the empirical size compared with Li and Yao's result in \cite{li2015homoscedasticity} under four different design distributions. We use $``{\rm CVT}"$ and $``{\rm FCVT}"$ to represent their test and our test respectively. The entries of design matrices are $i.i.d$ random samples generated  from $N(0,1)$, $t(1)$ ($t$ distribution with freedom degree 1), $F(3,2)$ ($F$ distribution with parameters 3 and 2) and logarithmic normal distribution respectively. The sample size $n$ is 512 and the dimension of covariates varies from 4 to 384.  We also follow \cite{dette1998testing} and consider the following two models:
\begin{description}
  \item[Model 1] $y_i=\mb x_i\mb \beta+\mb \varepsilon_i(1+\mb x_i \mb h),\ \ \ \ \ i=1,2,\cdots,n$, \\ where $\mb h=(1,\mb 0_{(p-1)})$,
  \item[Model 2] $y_i=\mb x_i\mb \beta+\mb \varepsilon_i(1+\mb x_i \mb h),\ \ \ \ \ i=1,2,\cdots,n $ \\
      where $\mb h=(\mb 1_{(p/2)},\mb 0_{(p/2)})$.
\end{description}
Tables \ref{table3} and \ref{table4} show the empirical power compared with Li and Yao's results under four different regressors distributions mentioned above.

Then, we consider the condition that the random error obey the two-point distribution. Specifically speaking, we suppose $P(\varepsilon_1=-1)=P(\varepsilon_1=1)=1/2$. Since Li and Yao's result is unapplicable in this situation, Table \ref{table5} just shows the empirical size and empirical power under Model 2 of our test under four different regressors  distributions mentioned above.

According to the simulation result, it is showed that when $p/n\to [0,1)$ as $n\to \infty$, our test always has good size and power under all regressors distributions.

\begin{table}[!hbp]
  \center
    \begin{tabular}{|c|cc|cc|cc|cc|cc|}
       \hline
         & N(0,1)& & t(1) & & $F(3,2)$ & & $e^{(N(5,3))}$ & \\
        \hline
        p & FCVT & CVT & FCVT & CVT & FCVT & CVT & FCVT& CVT\\
       \hline
       4 & 0.0582 & 0.0531 & 0.0600 & 0.0603 & 0.0594 & 0.0597 & 0.0590 & 0.0594\\
       \hline
       16 & 0.0621 & 0.0567 & 0.0585 & 0.0805 & 0.0585 & 0.0824 & 0.0595 & 0.0803\\
       \hline
       64 & 0.0574 & 0.0515 & 0.0605 & 0.2245 & 0.0586 & 0.2312 & 0.0578 & 0.2348\\
       \hline
       128 & 0.0597 & 0.0551 & 0.0597 & 0.5586 & 0.0568 & 0.5779 & 0.0590 & 0.5934\\
       \hline
       256 & 0.0551 & 0.0515 & 0.0620 & 0.9868 & 0.0576 & 0.9908 & 0.0595 & 0.9933\\
       \hline
       384 & 0.0580 & 0.0556 & 0.0595 & 1.0000 & 0.0600 & 1.0000 & 0.0600 & 1.0000\\
       \hline
     \end{tabular}
  \caption{empirical size under different distributions}\label{table2}
  \end{table}

  \begin{table}[!hbp]
  \center
    \begin{tabular}{|c|cc|cc|cc|cc|cc|}
       \hline
         & N(0,1)& & t(1) & & $F(3,2)$ & & $e^{(N(5,3))}$ & \\
        \hline
        p & FCVT & CVT & FCVT & CVT & FCVT & CVT & FCVT& CVT\\
       \hline
       4 & 1.0000 & 1.0000 & 1.0000 & 1.0000 & 1.0000 & 1.0000 & 1.0000 & 1.0000\\
       \hline
       16 & 1.0000 & 1.0000 & 1.0000 & 1.0000 & 1.0000 & 1.0000 & 1.0000 & 1.0000\\
       \hline
       64 & 1.0000 & 1.0000 & 1.0000 & 1.0000 & 1.0000 & 1.0000 & 1.0000 & 1.0000\\
       \hline
       128 & 1.0000 & 1.0000 & 1.0000 & 1.0000 & 1.0000 & 1.0000 & 1.0000 & 1.0000\\
       \hline
       256 & 1.0000 & 1.0000 & 1.0000 & 1.0000 & 1.0000 & 1.0000 & 1.0000 & 1.0000\\
       \hline
       384 & 0.8113 & 0.8072 & 0.9875 & 1.0000 & 0.9876 & 1.0000 & 0.9905 & 1.0000\\
       \hline
     \end{tabular}
  \caption{empirical power under model 1}\label{table3}
  \end{table}

  \begin{table}[!hbp]
  \center
    \begin{tabular}{|c|cc|cc|cc|cc|cc|}
       \hline
         & N(0,1)& & t(1) & & $F(3,2)$ & & $e^{(N(5,3))}$ & \\
        \hline
       p & FCVT & CVT & FCVT & CVT & FCVT & CVT & FCVT& CVT\\
       \hline
       4 & 1.0000 & 1.0000 & 1.0000 & 1.0000 & 1.0000 & 1.0000 & 1.0000 & 1.0000\\
       \hline
       16 & 1.0000 & 1.0000 & 1.0000 & 1.0000 & 1.0000 & 1.0000 & 1.0000 & 1.0000\\
       \hline
       64 & 1.0000 & 1.0000 & 1.0000 & 1.0000 & 1.0000 & 1.0000 & 1.0000 & 1.0000\\
       \hline
       128 & 1.0000 & 1.0000 & 1.0000 & 1.0000 & 1.0000 & 1.0000 & 1.0000 & 1.0000\\
       \hline
       256 &  1.0000 & 1.0000 & 1.0000 & 1.0000 & 1.0000 & 1.0000 & 1.0000 & 1.0000\\
       \hline
       384 &  0.9066 & 0.9034 & 0.9799 & 1.0000 & 0.9445 & 1.0000 & 0.8883 & 1.0000\\
       \hline
     \end{tabular}
  \caption{empirical power under model 2}\label{table4}
  \end{table}

\begin{table}[!hbp]
  \center
    \begin{tabular}{|c|cc|cc|cc|cc|cc|}
       \hline
         & N(0,1)& & t(1) & & $F(3,2)$ & & $e^{(N(5,3))}$ & \\
        \hline
        p & Size & Power & Size & Power & Size & Power & Size & Power \\
       \hline
       4 & 0.0695 & 1.0000 & 0.0726 & 1.0000 & 0.0726 & 1.0000 & 0.0664 & 1.0000\\
       \hline
       16 & 0.0695 &  1.0000 & 0.0638 & 1.0000 & 0.0706 & 1.0000 & 0.0556 & 1.0000\\
       \hline
       64 & 0.0646 &  1.0000 & 0.0606 & 1.0000 & 0.0649 & 1.0000 & 0.0622 & 1.0000\\
       \hline
       128 & 0.0617 &  1.0000 & 0.0705 & 1.0000 & 0.0597 & 1.0000 & 0.0630 & 1.0000\\
       \hline
       256 & 0.0684 &  1.0000 & 0.0685 & 1.0000 & 0.0608 & 1.0000 & 0.0649 & 1.0000\\
       \hline
       384 & 0.0610 &  0.8529 & 0.0748 & 1.0000 & 0.0758 & 1.0000 & 0.0742 & 1.0000\\
       \hline
     \end{tabular}
  \caption{empirical size and power under different distributions}\label{table5}
  \end{table}

\subsection{Two Real Rata Analysis}

\subsubsection{The Death Rate Data Set}

In \cite{mcdonald1973instabilities}, the authors fitted a multiple linear regression of the total age adjusted mortality rate on 15 other variables (the average annual precipitation, the average January temperature, the average July temperature, the size of the population older than 65, the number of members per household, the number of years of schooling for persons over 22, the number of households with fully equipped kitchens, the population per square mile, the size of the nonwhite population, the number of office workers, the number of families with an income less than \$3000, the hydrocarbon pollution index, the nitric oxide pollution index, the sulfur dioxide pollution index and the degree of atmospheric moisture). The number of observations is 60. To investigate whether the homoscedasticity assumption in this models is justified, we applied our test and got a p-value of 0.4994, which strongly supported the assumption of constant variability in this model since we use the one side test. The data set is available at \url{http://people.sc.fsu.edu/~jburkardt/datasets/regression/regression.html}.

\subsubsection{The 30-Year Conventional Mortgage Rate Data Set}

The 30-Year Conventional Mortgage Rate data \cite{Mortgage} contains the economic data information of USA from 01/04/1980 to 02/04/2000 on a weekly basis (1049 samples). The goal is to predict the 30-Year Conventional Mortgage Rate by other 15 features . We used a multiple linear regression to fit this data set and got a good result. The adjusted R-squared is 0.9986, the P value of the overall F-test is 0. Our homoscedasticity test reported a p-value 0.4439.

\section{Proof Of The Main Theorem}

This section is to prove the main theorem. The first step is to establish the asymptotic normality of $\mb T_1$, $\mb T_2$ and $\alpha\mb T_1+\beta\mb T_2$ with $\alpha^2+\beta^2 \neq 0$  by the moment convergence theorem. Next we will calculate the expectations, variances and covariance of the statistics $\mb T_1=\sum_{i=1}^n\hat{\varepsilon_i}^4$ and $\mb T_2=\frac{1}{n}\(\sum_{i=1}^n\hat{\varepsilon_i}^2\)^2$. The main theorem then follows by the delta method. Note that without loss of generality, under $H_0$, we can assume that $\sigma=1$.

\subsection{The asymptotic normality of the statistics.}\label{clt}

 We start by giving a definition in Graph Theory.


\begin{definition}
  A graph $\mb G=\(\mb V,\mb E,\mb F\)$ is called two-edge connected, if removing any one edge from $G$, the resulting subgraph is still connected.
\end{definition}

The next lemma is a fundamental theorem for Graph-Associated Multiple Matrices without the proof. For the details of this theorem, one can refer to the section $\mb A.4.2$ in \cite{bai2010spectral}.

\begin{lemma}\label{lm2}
 Suppose that $\mb G=\(\mb V,\mb E, \mb F\)$ is a two-edge connected graph with $t$ vertices and $k$ edges. Each vertex $i$ corresponds to an integer $m_i \geq 2$ and each edge $e_j$ corresponds to a matrix $\mb T^{(j)}=\(t_{\alpha,\beta}^{(j)}\),\ j=1,\cdots,k$, with consistent dimensions, that is, if $F(e_j)=(f_i(e_j),f_e(e_j))=(g,h),$ then the matrix $\mb T^{\(j\)}$ has dimensions $m_g\times m_h$. Define $\mb v=(v_1,v_2,\cdots,v_t)$ and
 \begin{align}
   T'=\sum_{\mb v}\prod_{j=1}^kt_{v_{f_i(e_j)},v_{f_e(e_j)}}^{(j)},
 \end{align}
 where the summation $\sum_{\mb v}$ is taken for $v_i=1,2,\cdots, m_i, \ i=1,2,\cdots,t.$ Then for any $i\leq t$,
 we have
 $$|T'|\leq m_i\prod_{j=1}^k\|\mb T^{(j)}\|.$$
\end{lemma}

Let $\mathcal{T}=(\mb T^{(1)},\cdots,\mb T^{(k)})$ and define $G(\mathcal{T})=(G,\mathcal{T})$ as a Graph-Associated Multiple Matrices. Write $T'=sum(G(\mathcal{T}))$, which is referred to as the summation of the corresponding Graph-Associated Multiple Matrices.

We also need the following truncation lemma

\begin{lemma}\label{lm3}
Suppose that $\xi_n=\(\xi_1,\cdots,\xi_n\)$ is an i.i.d sequence with $\E|\xi_1|^r \leq \infty$, then there exists a sequence of positive numbers $(\eta_1,\cdots,\eta_n)$ satisfy that as $n \to \infty$, $\eta_n \to 0$ and
 $$P(\xi_n\neq\widehat \xi_n,\ {\rm i.o.})=0,$$
 where
 $\widehat \xi_n=\(\xi_1 I(|\xi_1|\leq \eta_n n^{1/r}),\cdots,\xi_n I(|\xi_n|\leq \eta_n n^{1/r})\).$
 And the convergence rate of $\eta_n$ can be slower than any preassigned rate.
\end{lemma}

\begin{proof}
$\E|\xi_1|^r \leq \infty$ indicated that for any $\epsilon>0$, we have
$$\sum_{m=1}^\infty 2^{2m}P(|\xi_1|\geq\epsilon2^{2m/r})\leq \infty.$$

Then there exists a sequence of positive numbers $\epsilon=(\epsilon_1,\cdots,\epsilon_m)$ such that
$$\sum_{m=1}^\infty 2^{2m}P(|\xi_1|\geq\epsilon_m2^{2m/r})\leq \infty,$$
and $\epsilon_m \to 0$ as $m \to 0$. And the convergence rate of $\epsilon_m$ can be slower than any preassigned rate.

Now, define $\delta_n=2^{1/r}\epsilon_m$ for $2^{2m-1}\leq n\leq 2^{2m}$, we have as $n\to \infty$
\begin{align}
P(\xi_n\neq\widehat \xi_n,\ {\rm i.o.})\leq &\lim_{k\to \infty}\sum_{m=k}^{\infty}P\Big(\bigcup_{2^{2m-1}\leq n\leq 2^{2m}}\bigcup_{i=1}^n\(|\xi_i|\geq\eta_nn^{1/r}\)\Big)\\\notag
\leq&\lim_{k\to \infty}\sum_{m=k}^{\infty}P\Big(\bigcup_{2^{2m-1}\leq n\leq 2^{2m}}\bigcup_{i=1}^{2^{2m}}\(|\xi_i|\geq\epsilon_m 2^{1/r}2^{\frac{\(2m-1\)}{r}}\)\Big)\\\notag
\leq&\lim_{k\to \infty}\sum_{m=k}^{\infty}P\Big(\bigcup_{2^{2m-1}\leq n\leq 2^{2m}}\bigcup_{i=1}^{2^{2m}}\(|\xi_i|\geq\epsilon_m 2^{{2m}/{r}}\)\Big)\\\notag
=&\lim_{k\to \infty}\sum_{m=k}^{\infty}P\Big(\bigcup_{i=1}^{2^{2m}}\(|\xi_i|\geq\epsilon_m 2^{{2m}/{r}}\)\Big)\\\notag
\leq&\lim_{k\to \infty}\sum_{m=k}^{\infty}2^{2m}P\Big(|\xi_1|\geq\epsilon_m 2^{{2m}/{r}}\Big)=0.
\end{align}
\end{proof}

We note that the truncation will neither change the symmetry of the distribution of $\xi_1$ nor change the order of the variance of $\mb T$.

Now, we come to the proof of the asymptotic normality of the statistics. We below give the proof of the asymptotic normality of $\alpha\mb T_1+\beta\mb T_2$ , where $\alpha^2+\beta^2\neq 0$. The asymptotic normality of either $\mb T_1$ or $\mb T_2$ is a result of setting $\alpha=0$ or $\beta=0$ respectively.

Denote $\mu_1=\E \mb T_1=\E\sum_{i=1}^n\hat \varepsilon_i^4$, $\mu_2=\E \mb T_2=\E n^{-1}\(\sum_{i=1}^n\hat \varepsilon_i^2\)^2$ and $S=\sqrt{{\rm {Var}}\(\alpha \mb T_1+\beta\mb T_2\)}$. Below is devote to calculating the moments of
$\mb T_0=\frac{\alpha \mb T_1+\beta \mb T_2-\(\alpha \mu_1+\beta \mu_2\)}{S}=\frac{\alpha \(\mb T_1-\mu_1\)+\beta \(\mb T_2-\mu_2\)}{S}$.

Note that by Lemma \ref{lm3}, we can assume that $\xi_1$ is truncated at $\eta_n n^{1/8}$. Then we have for large enough $n$ and $l>4$,
$$M_{2l}\leq \eta_n M_8{\sqrt n}^{2l/4-1}. $$

Let's take a look at the random variable
\begin{align}
&\alpha T_1+\beta T_2=\alpha \sum_{i=1}^n\(\sum_{j=1}^n a_{ij}\xi_j\)^4+(n^{-1})\beta \(\sum_{i=1}^n\(\sum_{j=1}^n a_{ij}\xi_j\)^2\)^2\\\notag
=&\alpha \sum_{i,j_1,\cdots,j_4} a_{i,j_1}a_{i,j_2}a_{i,j_3}a_{i,j_4}\xi_{j_1}\xi_{j_2}\xi_{j_3}\xi_{j_4}+(n^{-1})\beta\sum_{i_1,i_2,j_1,\cdots,j_4} a_{i_1,j_1}a_{i_1,j_2}a_{i_2,j_3}a_{i_2,j_4}\xi_{j_1}\xi_{j_2}\xi_{j_3}\xi_{j_4}\\\notag
=&\alpha \sum_{i,j_1,\cdots,j_4} a_{i,j_1}a_{i,j_2}a_{i,j_3}a_{i,j_4}\xi_{j_1}\xi_{j_2}\xi_{j_3}\xi_{j_4}+(n^{-1})\beta\sum_{u_1,u_2,v_1,\cdots,v_4} a_{u_1,v_1}a_{u_1,v_2}a_{u_2,v_3}a_{u_2,v_4}\xi_{v_1}\xi_{v_2}\xi_{v_3}\xi_{v_4}.
\end{align}
We next construct two type of graphs for the last two sums.

For given integers $i,j_1,j_2,j_3,j_4\in [1,n]$, draw a graph as follows:  draw two parallel lines, called the $I$-line and the $J$-line respectively; plot $i$ on the $I$-line and $j_1,j_2,j_3$ and $j_4$ on the $J$-line; finally, we draw four edges from $i$ to $j_t$, $t=1,2,3,4$ marked with $\textcircled{1}$. Each edge $(i,j_t)$ represents the random variable $a_{i,j_t}\xi_{j_t}$ and the graph $G_1(i,\mb j)$ represents $\prod_{\rho=1}^{4}a_{i,j_\rho}\xi_{j_\rho}$. For any given integer $k_1$, we draw $k_1$ such graphs between the $I$-line and the $J$-line denoted by $G_1(\tau)=G_1(i_\tau,\mb j_\tau)$, and write $G_{(1,k_1)}=\cup_{\tau} G_1(\tau)$.

For given integers $u_1,u_2,v_1,v_2,v_3,v_4\in [1,n]$, draw a graph as follows: plot $u_1$ and $u_2$ on the $I$-line and $v_1,v_2,v_3$ and $v_4$ on the $J$-line; then, we draw two edges from $u_1$ to $v_1$ and $v_2$ marked with $\textcircled{2}$ , draw two edges from $u_2$ to $v_3$ and $v_4$ marked with $\textcircled{2}$. Each edge $(u_l,v_t)$ represents the random variable $a_{u_l,v_t}\xi_{v_t}$ and the graph $G_2(\mb u,\mb v)$ represents $a_{u_1,v_1}a_{u_1,v_2}a_{u_2,v_3}a_{u_2,v_4}$.
For any given integer $k_2$, we draw $k_2$ such graphs between the $I$-line and the $J$-line denoted by $G_2(\psi)=G_2(\mb u_\psi,\mb v_\psi)$, and write $G_{(2,k_2)}=\cup_{\psi} G_2(\psi)$, $G_{k}=G_{(1,k_1)}\cup G_{(2,k_2)}$. Then the $k$-th order moment of $\mb T_0$ is
\begin{align*}
M_k'=&S^{-k}\sum_{k_1+k_2=k}{k\choose k_1}\alpha^{k_1}\beta^{k_2}\sum_{\substack{\{i_1,\mb j_1,\cdots,i_{k_1},\mb j_{k_1}\} \\ \{\mb u_1,\mb v_1,\cdots,\mb u_{k_2},\mb v_{k_2}\}}}\\
&n^{-k_2}\E\Big[\prod_{\tau=1}^{k_1}[G_1(i_\tau,\mb j_\tau)-\E(G_1(i_\tau,\mb j_\tau))]\prod_{\phi=1}^{k_2}[G_2(\mb u_\psi,\mb v_\psi)-\E(G_2(\mb u_\psi,\mb v_\psi))]\Big].
\end{align*}

We first consider a graph $G_k$ for the given set of integers $k_1,k_2$, $i_1, \mb j_1,\cdots,i_{k_1},\mb j_{k_1}$ and $\mb u_1,\mb v_1,\cdots,\mb u_{k_2},\mb v_{k_2}$. We have the following simple observations:
Firstly, if $G_k$ contains a $j$ vertex of odd degree, then the term is zero because odd-ordered moments  of random variable $\xi_j$ are 0.
Secondly, if there is a subgraph $G_1(\tau)$ or $G_2(\psi)$ that does not have an $j$ vertex coinciding with any $j$ vertices of other subgraphs, the term is also 0 because $G_1(\tau)$ or $G_2(\psi)$ is independent of the remainder subgraphs.

Then, upon these two observations, we split the summation of non-zero terms in $M_k'$ into a sum of partial sums in accordance of isomorphic classes (two graphs are called isomorphic if one can be obtained from the other by a permutation of $(1,2,\cdots,n)$, and all the graphs are classified into isomorphic  classes. For convenience, we shall choose one graph from an isomorphic class as the canonical graph of that class). That is, we may write
\begin{align*}
M_k'=&S^{-k}\sum_{k_1+k_2=k}{k\choose k_1}\alpha^{k_1}\beta^{k_2}n^{-k_2}\sum_{G_k'}M_{G_k'},
\end{align*}
where
$$
M_{G_k'}=\sum_{G_k\in G_k'}\E G_k.
$$
Here $G_k'$ is a canonical graph and $\sum_{G_k\in G_k'}$ denotes the summation for all graphs $G_k$ isomorphic to $G_k'$.

In that follows, we need the fact that the variances of $\mb T_1$ and $\mb T_2$ and their covariance are all of order n.
This will be proved in Section \ref{var}.

Since all of the vertices in the non-zero canonical graphs have even degrees, every connected component of them is a circle, of course a two-edge connected graph. For a given isomorphic class with canonical graph $G_k'$, denote by $c_{G_k'}$ the number of connected components of the canonical graph $G_k'$. For every connected component $G_0$ that has $l$ non-coincident $J$-vertices with degrees $d_1,\cdots,d_l$, let $d'=\max\{d_1-8,\cdots,d_l-8,0\}$, denote $\mathcal{T}=(\underbrace{\mb A,\cdots,\mb A}_{\sum_{t=1}^l d_t})$ and define $G_0(\mathcal{T})=(G_0,\mathcal{T})$ as a Graph-Associated Multiple Matrices. By Lemma
\ref{lm2} we then conclude that the contribution of this canonical class is at most $\(\prod_{t=1}^l M_{d_t}\)sum(G(\mathcal{T}))=O(\eta_n^{d'}n\sqrt n^{d'/4})$. Noticing that $\eta_n \to 0$, if $c_{G_k'}$  is less than $k/2+k_2$, then the contribution of this canonical class is negligible because
$S^k\asymp n^{k/2}$ and $M_{G_k'}$ in $M_k'$ has a factor of $n^{-k_2}$. However one can see that $c_{G_k'}$ is at most $[k/2]+k_2$ for every ${G_k'}$ by the argument above and noticing that every $G_2(\bullet)$ has two $i$ vertices. Therefore, $M_k'\to 0$ if $k$ is odd.

Now we consider the limit of $M_k'$ when $k=2s$. We shall say that {\it the given set of integers $i_1, \mb j_1,\cdots,i_{k_1},\mb j_{k_1}$ and $\mb u_1,\mb v_1,\cdots,\mb u_{k_2},\mb v_{k_2}$ (or equivalent the graph $G_k$) satisfies the condition $c(s_1,s_2,s_3)$ if in the graph $G_k$ plotted by this set of integers 
there are $2s_1$ $G_1{\(\bullet\)}$ connected pairwisely, $2s_2$ $G_2{\(\bullet\)}$ connected pairwisely and $s_3$ $G_1{\(\bullet\)}$ connected with $s_3$ $G_2{\(\bullet\)}$, where $2s_1+s_3=k_1$, $2s_2+s_3=k_2$ and $s_1+s_2+s_3=s$, say $G_1{\(2\tau-1\)}$ connects $G_1{\(2\tau\)}$, $\tau=1,2,\cdots,s_1$, $G_2{\(2\psi-1\)}$ connects $G_1{\(2\psi\)}$, $\psi=1,2,\cdots,s_2$ and $G_1{\(2s_1+\varphi\)}$ connects $G_2{\(2s_2+\varphi\)}$, $\varphi=1,2,\cdots,s_3$, and there are no other connections between subgraphs.}
Then, for any $G_k$ satisfying $c(s_1,s_2,s_3)$, we have
\begin{align}
  \E G_k=&\prod_{\tau=1}^{s_1}\E[(G_1{\(2\tau-1\)}-\E(G_1{\(2\tau-1\)}))(G_1{\(2\tau\)}-\E(G_1\({2\tau}\)))]\times\\\notag
           &\prod_{\psi=1}^{s_2}\E[(G_2{\(2\psi-1\)}-\E(G_2{\(2\psi-1\)}))(G_2{\(2\psi\)}-\E(G_2\({2\psi}\)))]\times\\\notag
           &\prod_{\varphi=1}^{s_3}\E[(G_1{\(2s_1+\varphi\)}-\E(G_1{\(2s_1+\varphi\)}))(G_2{\(2s_2+\varphi\)}-\E(G_2\({2s_2+\varphi}\)))].
\end{align}

Now, we compare
\begin{align}
&n^{-k_2}\sum_{G_k\in c(s_1,s_2,s_3)} \E G_k\\\notag
=&n^{-k_2}\sum_{G_k\in c(s_1,s_2,s_3)}\prod_{\tau=1}^{s_1}\E[(G_1{\(2\tau-1\)}-\E(G_1{\(2\tau-1\)}))(G_1{\(2\tau\)}-\E(G_1\({2\tau}\)))]\times\\\notag
           &\prod_{\psi=1}^{s_2}\E[(G_2{\(2\psi-1\)}-\E(G_2{\(2\psi-1\)}))(G_2{\(2\psi\)}-\E(G_2\({2\psi}\)))]\times \\\notag
           &\prod_{\varphi=1}^{s_3}\E[(G_1{\(2s_1+\varphi\)}-\E(G_1{\(2s_1+\varphi\)}))(G_2{\(2s_2+\varphi\)}-\E(G_2\({2s_2+\varphi}\)))],
           \end{align}
with
\begin{align}
&\(\E\(\mb T_1-\mu_1\)^2\)^{s_1}\(\E\(\mb T_2-\mu_2\)^2\)^{s_2}\(\E\(\mb T_1-\mu_1\)\(\mb T_2-\mu_2\)\)^{s_3}\\\notag
=&n^{-k_2}\sum_{G_k}\prod_{\tau=1}^{s_1}\E[(G_1{\(2\tau-1\)}-\E(G_1{\(2\tau-1\)}))(G_1{\(2\tau\)}-\E(G_1\({2\tau}\)))]\times\\\notag
           &\prod_{\psi=1}^{s_2}\E[(G_2{\(2\psi-1\)}-\E(G_2{\(2\psi-1\)}))(G_2{\(2\psi\)}-\E(G_2\({2\psi}\)))]\times \\\notag
           &\prod_{\varphi=1}^{s_3}\E[(G_1{\(2s_1+\varphi\)}-\E(G_1{\(2s_1+\varphi\)}))(G_2{\(2s_2+\varphi\)}-\E(G_2\({2s_2+\varphi}\)))],
\end{align}
where $\sum_{G_k\in c(s_1,s_2,s_3)}$ stands for the summation running over all graph $G_k$ satisfying the condition $c(s_1,s_2,s_3)$.

If $G_k$ satisfies the two observations mentioned before, then $\E G_k=0$, which does not appear in both expressions; if $G_k$ satisfies the condition
$c(s_1,s_2,s_3)$, then the two expressions both contain $\E G_k$. Therefore, the second expression contains more terms that $G_k$ have more connections among  subgraphs than the condition $c(s_1,s_2,s_3)$. Therefore, by Lemma \ref{lm2},
\begin{equation}
\(\E\(\mb T_1-\mu_1\)^2\)^{s_1}\(\E\(\mb T_2-\mu_2\)^2\)^{s_2}\(\E\(\mb T_1-\mu_1\)\(\mb T_2-\mu_2\)\)^{s_3}=n^{-k_2}\sum_{G_k\in c(s_1,s_2,s_3)} \E G_k+o(S^k).
\label{map1}
\end{equation}

If $G_k\in G_k'$ with $c_{G_k'}=s+k_2$, for any nonnegative integers $s_1,s_2,s_3$ satisfying $k_1=2s_1+s_3$, $k_2=2s_2+s_3$ and $s_1+s_2+s_3=s$, we have ${k_1\choose s_3}{k_2 \choose s_3}(2s_1-1)!!(2s_2-1)!!s_3!$ ways to pairing the subgraphs satisfying the condition $c(s_1,s_2,s_3)$. By (\ref{map1}), we then have
\begin{eqnarray*}
&&\sum_{c_{G_k'}=s+k_2}n^{-k_2}\E G_k+o(S^k)\\
&=&\sum_{s_1+s_2+s_3=s\atop 2s_1+s_3=k_1, 2s_2+s_3=k_2}{k_1\choose s_3}{k_2 \choose s_3}(2s_1-1)!!(2s_2-1)!!s_3!
(Var(\mb T_1))^{s_1}(Var(\mb T_2))^{s_2}(Cov(\mb T_1,\mb T_2))^{s_3}
\end{eqnarray*}


It follows that
\begin{align*}
M_k'=&S^{-k}\sum_{k_1+k_2=k}{k\choose k_1}\alpha^{k_1}\beta^{k_2}n^{-k_2}\sum_{c_{G_k'}=s+k_2}\E G_k+o(1)\\
=&\Big(S^{-2s}\sum_{k_1=0}^{2s}\sum_{s_3=0}^{\min\{k_1,k_2\}}{2s\choose k_1}{k_1 \choose s_3}{k_2 \choose s_3}\(2s_1-1\)!!\(2s_2-1\)!!s_3!\\
&\(\alpha^2Var(\mb T_1)\)^{s_1}\(\beta^2Var(\mb T_2)\)^{s_2}\(\alpha\beta Cov(\mb T_1,\mb T_2)\)^{s_3}\Big)+o(1)\\
=&\Big(S^{-2s}\sum_{s_1+s_2+s_3=s}{2s\choose 2s_1+s_3}{2s_1+s_3 \choose s_3}{2s_2+s_3 \choose s_3}\(2s_1-1\)!!\(2s_2-1\)!!s_3!\\
&\(\alpha^2Var(\mb T_1)\)^{s_1}\(\beta^2Var(\mb T_2)\)^{s_2}\(\alpha\beta Cov(\mb T_1,\mb T_2)\)^{s_3}\Big)+o(1)\\
=&\Big(S^{-2s}\sum_{s_1+s_2+s_3=s}\frac{(2s)!(2s_1+s_3)!(2s_2+s_3)!}{\(2s_1+s_3\)!(2s_2+s_3)!s_3!(2s_1)!s_3!(2s_2)!}\(2s_1-1\)!!\(2s_2-1\)!!s_3!\\
&\(\alpha^2Var(\mb T_1)\)^{s_1}\(\beta^2Var(\mb T_2)\)^{s_2}\(\alpha\beta Cov(\mb T_1,\mb T_2)\)^{s_3}\Big)+o(1)\\
=&\Big(S^{-2s}\sum_{s_1+s_2+s_3=s}(2s-1)!!\frac{s!}{s_1!s_2!s_3!}\\
&\(\alpha^2Var(\mb T_1)\)^{s_1}\(\beta^2Var(\mb T_2)\)^{s_2}\(2\alpha\beta Cov(\mb T_1,\mb T_2)\)^{s_3}\Big)+o(1),
\end{align*}
which implies that
 $$
 M'_{2s}\to (2s-1)!!.
 $$

Combining the arguments above and the moment convergence theorem we conclude that
$$\frac{\mb T_1-\E \mb T_1}{\sqrt{{\rm Var}\mb T_1}}\stackrel{d}{\rightarrow} {\rm N}\(0,1\),\ \frac{\mb T_2-\E \mb T_2}{\sqrt{{\rm Var}\mb T_2}}\stackrel{d}{\rightarrow} {\rm N}\(0,1\), \ \frac{ \(\alpha \mb T_1+\beta \mb T_2\)-\E \(\alpha \mb T_1+\beta \mb T_2\)}{\sqrt{{\rm Var}\(\alpha \mb T_1+\beta \mb T_2\)}}\stackrel{d}{\rightarrow} {\rm N}\(0,1\),$$
where $\alpha^2+\beta^2\neq 0.$
Let $$\Sigma=\left(
             \begin{array}{cc}
               {\rm {Var}}(\mb T_1) & \rm {Cov}(\mb T_1,\mb T_2) \\
               \rm {Cov}(\mb T_1,\mb T_2) & \rm {Var}(\mb T_2) \\
             \end{array}
           \right)
.$$ We conclude that $\Sigma^{-1/2}\(\mb T_1-\E\mb T_1,\mb T_2-\E\mb T_2\)'$ is asymptotic two dimensional gaussian vector.

\subsection{The expectation}\label{exp}

In the following let $\mb B=\mb A\mb A'$. Recall that

  $$\mb{T_1}=\sum_{i=1}^n\widehat{\varepsilon_i}^4=\sum_{i=1}^n\(\sum_{j=1}^n a_{i,j}\xi_j\)^4=\sum_{i=1}^n\sum_{j_1,j_2,j_3,j_4}a_{i,j_1}a_{i,j_2}a_{i,j_3}a_{i,j_4}\xi_{j_1}\xi_{j_2}\xi_{j_3}\xi_{j_4},$$ $$\mb T_2=n^{-1}\(\sum_{i=1}^n\(\sum_{j=1}^n a_{i,j}\xi_j\)^2\)^2
=n^{-1}\sum_{i_1,i_2}\sum_{j_1,j_2,j_3,j_4}a_{i_1,j_1}a_{i_1,j_2}a_{i_2,j_3}a_{i_2,j_4}\xi_{j_1}\xi_{j_2}\xi_{j_3}\xi_{j_4}.$$
Since all  odd moments of $\xi_1,\cdots,\xi_n$ are 0, we know that $\E\mb T_1$ and $\E\mb T_2$ are only affected by terms whose multiplicities of distinct values in the sequence $(j_1,\cdots,j_4)$ are all even.

We need to evaluate the mixed moment $\E \(\mb T_1^{\gamma}\mb T_2^{\omega}\)$. For simplifying notations particularly in Section \ref{var} we introduce the following notations
\begin{align}
&\Omega_{\{\omega_1,\omega_2,\cdots,\omega_s\}}^{\(\gamma_1,\gamma_2,\cdots,\gamma_t\)}[\underbrace{\(\phi_{1,1},\cdots,\phi_{1,s}\),\(\phi_{2,1},\cdots,\phi_{2,s}\),
\cdots,\(\phi_{t,1},\cdots,\phi_{t,s}\)}_{t \ groups}]_0 \\
=&\sum_{i_1,\cdots,i_{t},j_1\neq\cdots\neq j_{s}}\prod_{\tau=1,\cdots, t}\prod_{\rho=1,\cdots, s} a_{i_{\tau},j_{\rho}}^{\phi_{\tau,\rho}},
\end{align}
where $i_1,\cdots,i_{t}$ and $j_1,\cdots,j_{s}$ run over $1,\cdots,n$  and are subject to the restrictions that $j_1,\cdots,j_s$ are distinct; $\sum_{l=1}^t \gamma_l=\sum_{l=1}^s \omega_l=\theta,$ and for any $k=1,\cdots, s$, $\sum_{l=1}^t \phi_{l,k}=\theta$.  Intuitively, $t$ is the number of distinct $i$-indices and $s$ that of distinct $j$'s; $\gamma_\tau$ is the multiplicity of the index $i_\tau$ and $\omega_\rho=\sum_{l=1}^t\phi_{l,\rho}$ that of $j_\rho$;  $\phi_{\tau,\rho}$ the multiplicity of the factor $a_{i_\tau,j_\rho}$; and $\theta=4(\gamma+\omega)$.

Define
\begin{align*}
&\Omega_{\{\omega_1,\omega_2,\cdots,\omega_s\}}^{\(\gamma_1,\gamma_2,\cdots,\gamma_t\)}[\underbrace{\(\phi_{1,1},\cdots,\phi_{1,s}\),\(\phi_{2,1},\cdots,\phi_{2,s}\),
\cdots,\(\phi_{t,1},\cdots,\phi_{t,s}\)}_{t \ groups}] \\
=&\sum_{i_1,\cdots,i_{t},j_1,\cdots, j_{s}}\prod_{\tau=1,\cdots, t}\prod_{\rho=1,\cdots, s} a_{i_{\tau},j_{\rho}}^{\phi_{\tau,\rho}}.
\end{align*}
The definition above is similar to that of
$$\Omega_{\{\omega_1,\omega_2,\cdots,\omega_s\}}^{\(\gamma_1,\gamma_2,\cdots,\gamma_t\)}[\underbrace{\(\phi_{1,1},\cdots,\phi_{1,s}\),\(\phi_{2,1},\cdots,\phi_{2,s}\),
\cdots,\(\phi_{t,1},\cdots,\phi_{t,s}\)}_{t \ groups}]_0$$
 without the restriction that the indices $j_1,\cdots,j_s$ are distinct from each other. To help understand these notations we demonstrate some examples as follows.
\begin{align*}
\Omega_{\{2,2,2,2\}}^{(4,4)}[(2,2,0,0),(0,0,2,2)]=\sum_{i_1,i_2,j_1,\cdots,j_4}a_{i_1,j_1}^2a_{i_1,j_2}^2a_{i_2,j_3}^2a_{i_2,j_3}^2,
\end{align*}

\begin{align*}
\Omega_{\{2,2,2,2\}}^{(4,4)}[(2,1,1,0),(0,1,1,2)]=\sum_{i_1,i_2,j_1,\cdots,j_4}a_{i_1,j_1}^2a_{i_1,j_2}a_{i_1,j_3}a_{i_2,j_2}a_{i_2,j_3}a_{i_2,j_4}^2,
\end{align*}

\begin{align*}
\Omega_{\{2,2,2,2\}}^{(4,4)}[(2,2,0,0),(0,0,2,2)]_0=\sum_{i_1,i_2,j_1\neq\cdots\neq j_4}a_{i_1,j_1}^2a_{i_1,j_2}^2a_{i_2,j_3}^2a_{i_2,j_3}^2,
\end{align*}

\begin{align*}
\Omega_{\{2,2,2,2\}}^{(4,4)}[(2,1,1,0),(0,1,1,2)]_0=\sum_{i_1,i_2,j_1\neq\cdots\neq j_4}a_{i_1,j_1}^2a_{i_1,j_2}a_{i_1,j_3}a_{i_2,j_2}a_{i_2,j_3}a_{i_2,j_4}^2.
\end{align*}




We further use $M_k$ to denote the $k$-th order moment of the error random variable. We also use $\mb C_{n}^k$ to denote the combinatorial number $n \choose k$. We then obtain

\begin{align}\label{e1}
&\E \mb T_1=\E\sum_{i=1}^n\sum_{j_1,j_2,j_3,j_4}a_{i,j_1}a_{i,j_2}a_{i,j_3}a_{i,j_4}\xi_{j_1}\xi_{j_2}\xi_{j_3}\xi_{j_4}\\\notag
=&M_4\Omega_{\{4\}}^{(4)}+M_2^2{\Omega_{\{2,2\}}^{(4)}}_0=M_4\Omega_{\{4\}}^{(4)}+\frac{\mb C_4^2}{2!}{\Omega_{\{2,2\}}^{(4)}}[\(2,0\),\(0,2\)]_0\\\notag
=&M_4\Omega_{\{4\}}^{(4)}+\frac{\mb C_4^2}{2!}\({\Omega_{\{2,2\}}^{(4)}}[\(2,0\),\(0,2\)]-\Omega_{\{4\}}^{(4)}\)\\\notag
=&\frac{\mb C_4^2}{2!}{\Omega_{\{2,2\}}^{(4)}}[\(2,0\),\(0,2\)]+\nu_4 \Omega_{\{4\}}^{(4)}
=3\sum_i\(\sum_{j}a_{i,j}^2\)^2+\nu_4\sum_{ij}a_{ij}^4\\\notag
=&3\sum_ib_{i,i}^2+\nu_4\sum_{ij}a_{ij}^4=3\tr\(\mb B\du \mb B\)+\nu_4\tr(\mb A\du \mb A)'(\mb A\du \mb A),
\end{align}
where $\nu_4=M_4-3$
and

\begin{align}\label{e2}
&\E\mb T_2=n^{-1}\E\sum_{i_1,i_2}\sum_{j_1,j_2,j_3,j_4}a_{i_1,j_1}a_{i_1,j_2}a_{i_2,j_3}a_{i_2,j_4}\xi_{j_1}\xi_{j_2}\xi_{j_3}\xi_{j_4}\\\notag
=&n^{-1}\(M_4\Omega_{\{4\}}^{(2,2)}+M_2^2{\Omega_{\{2,2\}}^{(2,2)}}_0\)\\\notag
=&n^{-1}\(M_4\Omega_{\{4\}}^{(2,2)}+\({\Omega_{\{2,2\}}^{(2,2)}}[\(2,0\),\(0,2\)]_0+2\Omega_{\{2,2\}}^{(2,2)}[\(1,1\),\(1,1\)]_0\)\)\\\notag
=&n^{-1}\(M_4\Omega_{\{4\}}^{(2,2)}+\({\Omega_{\{2,2\}}^{(2,2)}}[\(2,0\),\(0,2\)]+2\Omega_{\{2,2\}}^{(2,2)}[\(1,1\),\(1,1\)]\)-3\Omega_{\{4\}}^{(2,2)}\)\\\notag
=&n^{-1}\(\sum_{i_1,i_2,j_1,j_2}a_{i_1,j_1}^2a_{i_2,j_2}^2+2\sum_{i_1,i_2,j_1,j_2}a_{i_1,j_1}a_{i_1,j_2}a_{i_2,j_1}a_{i_2,j_2}
+\nu_4\sum_{i_1,i_2,j}a_{i_1j}^2a_{i_2j}^2\)\\\notag
=&n^{-1}\(\(\sum_{i,j}a_{i,j}^2\)^2+2\sum_{i_1,i_2}\(\sum_{j}a_{i_1,j}a_{i_2,j}\)^2+\nu_4\sum_{i_1,i_2,j}a_{i_1j}^2a_{i_2j}^2\)\\\notag
=&n^{-1}\(\(\sum_{i,j}a_{i,j}^2\)^2+2\sum_{i_1,i_2}b_{i_1,i_2}^2+\nu_4\sum_{j=1}^nb_{jj}^2\)\\ \notag
=&n^{-1}\(\(\tr \mb B\)^2+2\tr\mb B^2+\nu_4\tr (\mb B\du \mb B)\).
\end{align}

\subsection{The variances and covariance}\label{var}

We are now in the position to calculate the variances of $\mb T_1$, $\mb T_2$ and their covariance.

First, we have

\begin{align}\label{t10}
&{\rm Var}( \mb T_1)=\E\(\sum_{i}\widehat{\varepsilon_i}^4-\E\(\sum_{i}\widehat{\varepsilon_i}^4\)\)^2\\\notag
=&\sum_{i_1,i_2,j_1,\cdots,j_8}[\E G(i_1,\mb j_1)G(i_2,\mb j_2)-\E G(i_1,\mb j_1)\E G(i_2,\mb j_2)]\\\notag
=&\Bigg(\Omega_{\{8\}}^{(4,4)}+\Omega_{\{2,6\}_0}^{(4,4)}+\Omega_{\{4,4\}_0}^{(4,4)}+\Omega_{\{2,2,4\}_0}^{(4,4)}+\Omega_{\{2,2,2,2\}_0}^{(4,4)}\Bigg),\\\notag
\end{align}
where the first term comes from the graphs in which the 8 $J$-vertices coincide together; the second term comes from the graphs
in which there are 6 $J$-vertices coincident and another two coincident and so on.

Because $G(i_1,\mb j_1)$ and $G(i_2,\mb j_2)$ have to connected each other, thus, we have
\begin{align}\label{t11}
&\Omega_{\{2,2,2,2\}_0}^{(4,4)}\\\notag=&\frac{\mb C_4^2\mb C_4^2\mb C_2^1\mb C_2^1}{2!}\Omega_{\{2,2,2,2\}}^{(4,4)}[(2,1,1,0),(0,1,1,2)]_0+{\mb C_4^1\mb C_3^1\mb C_2^1}\Omega_{\{2,2,2,2\}}^{(4,4)}[(1,1,1,1),(1,1,1,1)]_0\\\notag
=&72\Big(\Omega_{\{2,2,2,2\}}^{(4,4)}[(2,1,1,0),(0,1,1,2)]-4\Omega_{\{2,2,4\}}^{(4,4)}[(2,1,1),(0,1,3)]_0-\Omega_{\{2,2,4\}}^{(4,4)}[(2,0,2),(0,2,2)]_0\\\notag
 &-\Omega_{\{2,2,4\}}^{(4,4)}[(1,1,2),(1,1,2)]_0-2\Omega_{\{4,4\}}^{(4,4)}[(3,1),(1,3)]_0-\Omega_{\{4,4\}}^{(4,4)}[(2,2),(2,2)]_0\\\notag
 &-2\Omega_{\{2,6\}}^{(4,4)}[(1,3),(1,3)]_0-2\Omega_{\{2,6\}}^{(4,4)}[(2,2),(0,4)]_0-\Omega_{\{8\}}^{(4,4)}\Big)\\\notag
 &+24\Big(\Omega_{\{2,2,2,2\}}^{(4,4)}[(1,1,1,1),(1,1,1,1)]-6\Omega_{\{2,2,4\}}^{(4,4)}[(1,1,2),(1,1,2)]_0-3\Omega_{\{4,4\}}^{(4,4)}[(2,2),(2,2)]_0\\\notag
 &-4\Omega_{\{2,6\}}^{(4,4)}[(1,3),(1,3)]_0-\Omega_{\{8\}}^{(4,4)}\Big).\\\notag
=&72\Big(\Omega_{\{2,2,2,2\}}^{(4,4)}[(2,1,1,0),(0,1,1,2)]-4\Omega_{\{2,2,4\}}^{(4,4)}[(2,1,1),(0,1,3)]-\Omega_{\{2,2,4\}}^{(4,4)}[(2,0,2),(0,2,2)]\\\notag
 &-\Omega_{\{2,2,4\}}^{(4,4)}[(1,1,2),(1,1,2)]+2\Omega_{\{4,4\}}^{(4,4)}[(3,1),(1,3)]_0+\Omega_{\{4,4\}}^{(4,4)}[(2,2),(2,2)]_0\\\notag
 &+4\Omega_{\{2,6\}}^{(4,4)}[(1,3),(1,3)]_0+4\Omega_{\{2,6\}}^{(4,4)}[(2,2),(0,4)]_0+5\Omega_{\{8\}}^{(4,4)}\Big)\\\notag
 &+24\Big(\Omega_{\{2,2,2,2\}}^{(4,4)}[(1,1,1,1),(1,1,1,1)]-6\Omega_{\{2,2,4\}}^{(4,4)}[(1,1,2),(1,1,2)]+3\Omega_{\{4,4\}}^{(4,4)}[(2,2),(2,2)]_0\\\notag
 &+8\Omega_{\{2,6\}}^{(4,4)}[(1,3),(1,3)]_0+5\Omega_{\{8\}}^{(4,4)}\Big).\\\notag
=&72\Big(\Omega_{\{2,2,2,2\}}^{(4,4)}[(2,1,1,0),(0,1,1,2)]-4\Omega_{\{2,2,4\}}^{(4,4)}[(2,1,1),(0,1,3)]-\Omega_{\{2,2,4\}}^{(4,4)}[(2,0,2),(0,2,2)]\\\notag
 &-\Omega_{\{2,2,4\}}^{(4,4)}[(1,1,2),(1,1,2)]+2\Omega_{\{4,4\}}^{(4,4)}[(3,1),(1,3)]+\Omega_{\{4,4\}}^{(4,4)}[(2,2),(2,2)]\\\notag
 &+4\Omega_{\{2,6\}}^{(4,4)}[(1,3),(1,3)]+4\Omega_{\{2,6\}}^{(4,4)}[(2,2),(0,4)]-6\Omega_{\{8\}}^{(4,4)}\Big)\\\notag
 &+24\Big(\Omega_{\{2,2,2,2\}}^{(4,4)}[(1,1,1,1),(1,1,1,1)]-6\Omega_{\{2,2,4\}}^{(4,4)}[(1,1,2),(1,1,2)]+3\Omega_{\{4,4\}}^{(4,4)}[(2,2),(2,2)]\\\notag
 &+8\Omega_{\{2,6\}}^{(4,4)}[(1,3),(1,3)]-6\Omega_{\{8\}}^{(4,4)}\Big).
\end{align}

Likewise we have

\begin{align}\label{t12}
&{\Omega_{\{2,2,4\}}^{(4,4)}}_0\\\notag
=&{\mb C_2^1\mb C_4^3\mb C_4^1\mb C_3^1}M_4\Omega_{\{2,2,4\}}^{(4,4)}[(1,2,1),(1,0,3)]_0\\\notag
&+\frac{\mb C_4^2\mb C_4^2\mb C_2^1\mb C_2^1}{2!}M_4\Omega_{\{2,2,4\}}^{(4,4)}[(1,1,2),(1,1,2)]_0+{\mb C_4^2\mb C_4^2}(M_4-1)\Omega_{\{2,2,4\}}^{(4,4)}[(2,0,2),(0,2,2)]_0\\\notag
=&96M_4\Omega_{\{2,2,4\}}^{(4,4)}[(1,2,1),(1,0,3)]_0\\\notag
 &+72M_4\Omega_{\{2,2,4\}}^{(4,4)}[(1,1,2),(1,1,2)]_0+36(M_4-1)\Omega_{\{2,2,4\}}^{(4,4)}[(2,0,2),(0,2,2)]_0,\\\ \notag
=&96M_4\Omega_{\{2,2,4\}}^{(4,4)}[(1,2,1),(1,0,3)]\\\notag
 &+72M_4\Omega_{\{2,2,4\}}^{(4,4)}[(1,1,2),(1,1,2)]+36(M_4-1)\Omega_{\{2,2,4\}}^{(4,4)}[(2,0,2),(0,2,2)]\\\ \notag
 &-96M_4\Omega_{\{4,4\}}^{(4,4)}[(3,1),(1,3)]_0-(108 M_4-36)\Omega_{\{4,4\}}^{(4,4)}[(2,2),(2,2)]_0\\\notag
 &-240M_4\Omega_{\{2,6\}}^{(4,4)}[(1,3),(1,3)]_0-(168M_4-72)\Omega_{\{2,6\}}^{(4,4)}[(2,2),(0,4)]_0-(204M_4-36)\Omega_{\{8\}}^{(4,4)}\\\notag
 =&96M_4\Omega_{\{2,2,4\}}^{(4,4)}[(1,2,1),(1,0,3)]\\\notag
 &+72M_4\Omega_{\{2,2,4\}}^{(4,4)}[(1,1,2),(1,1,2)]+36(M_4-1)\Omega_{\{2,2,4\}}^{(4,4)}[(2,0,2),(0,2,2)]\\\ \notag
 &-96M_4\Omega_{\{4,4\}}^{(4,4)}[(3,1),(1,3)]-(108 M_4-36)\Omega_{\{4,4\}}^{(4,4)}[(2,2),(2,2)]\\\notag
 &-240M_4\Omega_{\{2,6\}}^{(4,4)}[(1,3),(1,3)]-(168M_4-72)\Omega_{\{2,6\}}^{(4,4)}[(2,2),(0,4)]+(408M_4-72)\Omega_{\{8\}}^{(4,4)}
\end{align}

\begin{align}\label{t13}
\Omega_{\{4,4\}_0}^{(4,4)}=&{\mb C_2^1\mb C_4^1\mb C_4^3}M_4^2\Omega_{\{4,4\}}^{(4,4)}[(3,1),(1,3)]_0
+\frac{\mb C_4^2\mb C_4^2}{2!}(M_4^2-1)\Omega_{\{4,4\}}^{(4,4)}[(2,2),(2,2)]_0\\\notag
=&16M_4^2\Omega_{\{4,4\}}^{(4,4)}[(3,1),(1,3)]+18(M_4^2-1)\Omega_{\{4,4\}}^{(4,4)}[(2,2),(2,2)]\\\notag
&-(34M_4^2-18)\Omega_{\{8\}}^{(4,4)},
\end{align}

\begin{align}\label{t14}
\Omega_{\{2,6\}_0}^{(4,4)}=&\mb C_2^1\mb C_4^2(M_6-M_4)\Omega_{\{2,6\}}^{(4,4)}[(2,2),(0,4)]_0+\mb C_4^1\mb C_4^1M_6\Omega_{\{2,6\}}^{(4,4)}[(1,3),(1,3)]_0\\\notag
=&12(M_6-M_4)\Omega_{\{2,6\}}^{(4,4)}[(2,2),(0,4)]+16M_6\Omega_{\{2,6\}}^{(4,4)}[(1,3),(1,3)]\\\notag
& -(28 M_6-12 M_4)\Omega_{\{8\}}^{(4,4)}.
\end{align}
 and
\begin{align}\label{t15}
\Omega_{\{8\}_0}^{(4,4)}=&(M_8-M_4^2)\Omega_{\{8\}}^{(4,4)}[(4),(4)].
\end{align}

Combining (\ref{t10}), (\ref{t11}), (\ref{t12}), (\ref{t13}), (\ref{t14}) and (\ref{t15}), we obtain

\begin{align}\label{vt1}
&{\rm Var}( \mb T_1)=72\Omega_{\{2,2,2,2\}}^{(4,4)}[(2,1,1,0),(0,1,1,2)]+24\Omega_{\{2,2,2,2\}}^{(4,4)}[(1,1,1,1),(1,1,1,1)]\\\notag
&+96(M_4-3)\Omega_{\{2,2,4\}}^{(4,4)}[(2,1,1),(0,1,3)]
+36(M_4-3)\Omega_{\{2,2,4\}}^{(4,4)}[(2,0,2),(0,2,2)]\\\notag
&+72(M_4-3)\Omega_{\{2,2,4\}}^{(4,4)}[(1,1,2),(1,1,2)]+16(M_4^2-6M_4+9)\Omega_{\{4,4\}}^{(4,4)}[(3,1),(1,3)]\\\notag
&+18(M_4^2-6M_4+9)\Omega_{\{4,4\}}^{(4,4)}[(2,2),(2,2)]+16(M_6-15M_4+30)\Omega_{\{2,6\}}^{(4,4)}[(1,3),(1,3)]\\\notag
&+12(M_6-15M_4+30)\Omega_{\{2,6\}}^{(4,4)}[(2,2),(0,4)]
+(M_8-28M_6-35M_4^2+420M_4-630)\Omega_{\{8\}}^{(4,4)}[(4),(4)],
\end{align}

where

\begin{align}
&\Omega_{\{2,2,2,2\}}^{(4,4)}[(2,1,1,0),(0,1,1,2)]=\sum_{i_1,\cdots,i_2,j_1,\cdots, j_4}a_{i_1,j_1}^2a_{i_1,j_2}a_{i_1,j_3}a_{i_2,j_2}a_{i_2,j_3}a_{i_2,j_4}^2\\\notag
&=\di'(\mb B) \(\mb B\du\mb B\)\di (\mb B),
\end{align}

\begin{align}
&\Omega_{\{2,2,2,2\}}^{(4,4)}[(1,1,1,1),(1,1,1,1)]=\sum_{i_1,\cdots,i_2,j_1,\cdots, j_4}a_{i_1,j_1}a_{i_1,j_2}a_{i_1,j_3}a_{i_1,j_4}a_{i_2,j_1}a_{i_2,j_2}a_{i_2,j_3}a_{i_2,j_4}\\\notag
&=\tr\(\mb B\du\mb B\)^2,
\end{align}

\begin{align}
&\Omega_{\{2,2,4\}}^{(4,4)}[(2,1,1),(0,1,3)]=\sum_{i_1,\cdots,i_2,j_1,\cdots, j_4}a_{i_1,j_1}^2a_{i_1,j_2}a_{i_1,j_3}a_{i_2,j_2}a_{i_2,j_3}^3=\tr\mb B \mb {D_B} \mb A \mb A'^{\du 3},
\end{align}

\begin{align}
&\Omega_{\{2,2,4\}}^{(4,4)}[(2,0,2),(0,2,2)]=\sum_{i_1,\cdots,i_2,j_1,\cdots, j_4}a_{i_1,j_1}^2a_{i_1,j_3}^2a_{i_2,j_2}^2a_{i_2,j_3}^2\\\notag
&=\di'(\mb B) \(\mb A\du\mb A\)\(\mb A\du\mb A\)'\di (\mb B) ,
\end{align}

\begin{align}
&\Omega_{\{2,2,4\}}^{(4,4)}[(1,1,2),(1,1,2)]=\sum_{i_1,\cdots,i_2,j_1,\cdots, j_4}a_{i_1,j_1}a_{i_1,j_2}a_{i_1,j_3}^2a_{i_2,j_1}a_{i_2,j_2}a_{i_2,j_3}^2\\\notag
&=\tr\(\(\mb B\du \mb B\)\(\mb A\du \mb A\)\(\mb A\du \mb A\)'\) ,
\end{align}

\begin{align}
&\Omega_{\{4,4\}}^{(4,4)}[(3,1),(1,3)]=\sum_{i_1,\cdots,i_2,j_1,\cdots, j_4}a_{i_1,j_1}^3a_{i_1,j_2}a_{i_2,j_1}a_{i_2,j_2}^3=\tr\(\(\mb A^{\du 3}\mb A'\)\(\mb A^{\du 3}\mb A'\)'\) ,
\end{align}

\begin{align}
&\Omega_{\{4,4\}}^{(4,4)}[(2,2),(2,2)]=\sum_{i_1,\cdots,i_2,j_1,\cdots, j_4}a_{i_1,j_1}^2a_{i_1,j_2}^2a_{i_2,j_1}^2a_{i_2,j_2}^2=\tr\(\(\mb A \du \mb A\)\(\mb A \du \mb A\)'\)^2 ,
\end{align}

\begin{align}
&\Omega_{\{2,6\}}^{(4,4)}[(1,3),(1,3)]=\sum_{i_1,\cdots,i_2,j_1,\cdots, j_4}a_{i_1,j_1}a_{i_1,j_2}^3a_{i_2,j_1}a_{i_2,j_2}^3=\tr\(\mb B \mb A^{\du 3} \mb A'^{\du 3}\) ,
\end{align}

\begin{align}
&\Omega_{\{2,6\}}^{(4,4)}[(2,2),(0,4)]=\sum_{i_1,\cdots,i_2,j_1,\cdots, j_4}a_{i_1,j_1}^2a_{i_1,j_2}^2a_{i_2,j_2}^4=\tr\(\(\mb A' \mb D_{\mb B}\mb A \)\du \(\mb A'^{\du 2}\mb A^{\du 2}\)\),
\end{align}

and

\begin{align}
&\Omega_{\{8\}}^{(4,4)}[(4),(4)]=\sum_{i_1,\cdots,i_2,j_1,\cdots, j_4}a_{i_1,j_1}^42a_{i_2,j_1}^4=\mb 1'\mb A^{\du 4}\mb A'^{\du 4}\mb 1,
\end{align}
%
%
%
%
%
%
Using the same procedure, we have

\begin{align}
&{\rm Var}(\mb T_2)=n^{-2}\(\E\(\sum_{i}\widehat{\varepsilon_i}^2\)^4-\E^2\(\sum_{i}\widehat{\varepsilon_i}^2\)^2\)\\\notag
=&n^{-2}\sum_{i_1,\cdots,i_4,j_1,\cdots,j_8}a_{i_1,j_1}a_{i_1,j_2}a_{i_2,j_3}a_{i_2,j_4}a_{i_3,j_5}a_{i_3,j_6}a_{i_4,j_7}a_{i_4,j_8}
\(\E\prod_{t=1}^8\xi_{j_t}-\E\prod_{t=1}^4\xi_{j_t}\E\prod_{t=5}^8\xi_{j_t}\)\\\notag
=&n^{-2}(P_{2,1}+P_{2,2})+O(1),
\end{align}
where

%

\begin{align}
P_{2,1}= \mb C_2^1\mb C_2^1\mb C_2^1\sum_{i_1,\cdots,i_4,j_1,\cdots, j_4}a^2_{i_1,j_1}a_{i_2,j_2}a_{i_3,j_2}a_{i_2,j_3}a_{i_3,j_3}a_{i_4,j_4}^2
=8\(\tr\mb B\)^2\tr\mb B^2,
\end{align}

\begin{align}
P_{2,2}=\nu_4\mb C_2^1\mb C_2^1\sum_{i_1,\cdots,i_4,j_1,j_2, j_3}a^2_{i_1,j_1}a^2_{i_2,j_2}a^2_{i_3,j_2}a_{i_4,j_3}^2
=4\nu_4\tr(\mb B'\du\mb B')\(\tr\mb B\)^2,
\end{align}

%
%
%
Similarly, we have

\begin{align}
&{\rm Cov}(\mb T_1,\mb T_2)=n^{-1}\(\E\(\sum_{i}\widehat{\varepsilon_i}^2\)^2\sum_{i}\widehat{\varepsilon_i}^4-\E\(\sum_{i}\widehat{\varepsilon_i}^2\)^2\sum_{i}\widehat{\varepsilon_i}^4\)\\\notag
=&n^{-1}\sum_{i_1,\cdots,i_3,j_1,\cdots,j_8}a_{i_1,j_1}a_{i_1,j_2}a_{i_2,j_3}a_{i_2,j_4}a_{i_3,j_5}a_{i_3,j_6}a_{i_3,j_7}a_{i_3,j_8}
\(\E\prod_{t=1}^8\xi_{j_t}-\E\prod_{t=1}^4\xi_{j_t}\E\prod_{t=5}^8\xi_{j_t}\)\\
=&n^{-1}(P_{3.1}+P_{3,2}+P_{3,3}+P_{3,4})+O(1),
\end{align}

where

%


\begin{align}
P_{3,1}=\mb C_4^2\mb C_2^1\mb C_2^1\sum_{i_1,\cdots,i_3,j_1,\cdots, j_4}a_{i_1,j_1}^2a_{i_2,j_2}a_{i_2,j_3}a_{i_3,j_2}a_{i_3,j_3}a_{i_3,j_4}^2
=24\tr\(\mb B^2\du \mb B\)\tr\mb B,
\end{align}

\begin{align}
P_{3,2}=\nu_4\mb C_4^1\mb C_2^1\mb C_2^1\sum_{i_1,\cdots,i_3,j_1,\cdots, j_3}a_{i_1,j_1}^2a_{i_2,j_2}a_{i_3,j_2}a_{i_2,j_3}a_{i_3,j_3}^3
=16\nu_4\tr(\mb B\mb A \mb A'^{\du 3})\tr\mb B,
\end{align}

\begin{align}
P_{3,3}=\nu_4\mb C_4^2\mb C_2^1\sum_{i_1,\cdots,i_3,j_1,\cdots, j_3}a_{i_1,j_1}^2a^2_{i_2,j_2}a^2_{i_3,j_2}a_{i_3,j_3}^2
=12\nu_4\tr\(\(\mb A'\mb D_{\mb B}\mb A\)\du \(\mb A'\mb A\)\)\tr\mb B,
\end{align}
\begin{align}\label{ct12}
P_{3,3}=\nu_6\mb C_2^1\sum_{i_1,i_2,i_3,j_1, j_2}a_{i_1,j_1}^2a^2_{i_2,j_2}a^4_{i_3,j_2}
=2\nu_6[\di(\mb A'\mb A)'(\mb A'^{\du 4})\mb 1]\tr\mb B,
\end{align}

%
%

We would like to point out that we do not need the assumption that $H_0$ holds up to now. From now on, in order to simplify the above formulas we assume $H_0$ holds.
%
%
%

Summarizing the calculations above, we obtain under $H_0$

\begin{align}
\E \mb T_1=3\sum_ib_{i,i}^2+\nu_4\sum_{ij}a_{ij}^4=3\tr\(\mb P\du \mb P\)+\nu_4\tr(\mb P\du \mb P)^2,
\end{align}

\begin{align}\label{e2}
\E\mb T_2=n^{-1}\(\(n-p\)^2+2\(n-p\)+\nu_4\tr (\mb P\du \mb P)\),
\end{align}

\begin{align}
{\rm Var}\mb T_1=&72\di'(\mb P) \(\mb P\du\mb P\)\di (\mb P)+24\tr\(\mb P\du\mb P\)^2\\\notag
&+\nu_4\(96\tr\mb P \mb {D_P} \mb P \mb P^{\du 3}+72\tr(\mb P\du \mb P)^3+36\di'(\mb P) \(\mb P\du\mb P\)^2\di (\mb P) \)\\\notag
&+\nu^2_4\( 18\tr(\mb P\du\mb P)^4+16\tr (\mb P^{\du 3}\mb P)^2)\)\\\notag
&+\nu_6\(12\tr\(\(\mb P\mb D_{\mb P}\mb P \)\du\(\mb P^{\du 2}\mb P^{\du 2}\)\)+16\tr \mb P \mb P^{\du 3}\mb P^{\du 3}\)+\nu_8\mb 1'(\mb P^{\du 4}\mb P^{\du 4})\mb 1,
\end{align}
\begin{align}
{\rm Var}(\mb T_2)=\frac{8\(n-p\)^3+4\nu_4\(n-p\)^2\tr(\mb P\du\mb P)}{n^2}+O(1),
\end{align}
and
\begin{align}
&{\rm Cov}(\mb T_1,\mb T_2)\\\notag
=&\frac{\(n-p\)}{n}\(24\tr\(\mb P\du \mb P\)+16\nu_4\tr(\mb P \mb P^{\du 3})+12\nu_4\tr\(\(\mb P\mb D_{\mb p}\mb P\)\du \mb P\)
+2\nu_6[\di(\mb P)'(\mb P^{\du 4})\mb 1]\).
\end{align}

\subsection{The proof of the main theorem}
Define a function $f(x,y)=\frac{x}{y}-1$. One may verify that $f_x(x,y)=\frac{1}{y},$ $f_y(x,y)=-\frac{x}{y^2}$, where $f_x(x,y)$ and $f_y(x,y)$ are the first order partial derivative. Since $\mb T=\frac{\mb T_1}{\mb T_2}-1,$ using the delta method, we have under $H_0$, $$\E \mb T=f(\E \mb T_1,\E\mb T_2)=\(\frac{3n\tr\(\mb P\du\mb P\)}{(n-p)^2+2\(n-p\)}-1\),$$
\begin{align}
  {\rm {Var}} \mb T=(f_x(\E\mb T_1,\E\mb T_2),f_y(\E\mb T_1,\E\mb T_2))\Sigma(f_x(\E\mb T_1,\E\mb T_2),f_y(\E\mb T_1,\E\mb T_2))'.
\end{align}
The proof of the main theorem is complete.


\end{document}